\documentclass[12pt]{article}
\usepackage{amsmath, amssymb, amsthm}
\usepackage{enumerate}
\usepackage{graphicx}
\usepackage{multirow}
\usepackage{array}
\usepackage{booktabs}
\usepackage{caption}
\usepackage{hyperref}
\usepackage{mathrsfs}
\usepackage{cite}
\usepackage{geometry}

\theoremstyle{definition}
\newtheorem{theorem}{Theorem}[section]
\newtheorem{lemma}[theorem]{Lemma}

\newtheorem{proposition}[theorem]{Proposition}
\newtheorem{remark}[theorem]{Remark}
\begin{document}

\title{On the Number of Connected Edge Cover Sets of Some Graph Families}
\author{Ali Zeydi Abdian$^1$, Saeid Alikhani$^2$, Mahsa Zare$^2$}
\maketitle
\begin{center}
	$^1$Department of Computer Science, Shahid Bahonar University of Kerman, Kerman, Iran\\
	{\tt alizeydiabdian@math.uk.ac.ir}
	
	\medskip
	$^{2}$Department of Mathematical Sciences, Yazd University, Yazd, Iran\\
	{\tt alikhani@yazd.ac.ir, ~~~ zare.mahsa@stu.yazd.ac.ir}
\end{center}
\date{}

\maketitle

\begin{abstract}
    Let $G=(V,E)$ be a simple connected graph. A connected edge cover of $G$ is a subset $S\subseteq E$ such that every vertex of $G$ is incident with at least one edge in $S$ and the subgraph induced by $S$ is connected. The connected edge cover polynomial of $G$ is defined as $E_c(G,x)=\sum_{i} e_c(G,i)x^i$, where $e_c(G,i)$ denotes the number of connected edge covers of $G$ with exactly $i$ edges. In this paper, we derive explicit formulas for both the connected edge cover polynomials and the total number of connected edge covers for several important graph families, including wheels, complete graphs $K_n$, complete bipartite graphs $K_{2,n}$, friendship graphs, and lollipop graphs. Each formula is accompanied by a combinatorial proof and verified by computational enumeration for small orders.
\end{abstract}

\section{Introduction}
\label{sec:intro}

Counting combinatorial structures in graphs is one of the central themes of modern graph theory and combinatorics. 
Enumerative graph parameters not only measure structural complexity but also provide deep connections to algebraic, probabilistic, and algorithmic aspects of graphs. 
Classical examples include the number of proper vertex colorings, encoded by the chromatic polynomial; the number of independent sets; the number of matchings; the number of dominating sets; and the number of edge covers. 
Each of these counting functions has generated extensive literature and has found applications in statistical physics, network reliability, coding theory, and theoretical computer science. For example see \cite{Akbari,Euro,Dong}. 

The study of counting parameters often reveals significantly more refined structural information than the corresponding extremal parameters. 
For instance, while the chromatic number determines the minimum number of colors needed for a proper coloring, the chromatic polynomial counts all such colorings and reflects subtle algebraic properties of the graph. 
Similarly, instead of merely determining the minimum size of a dominating set or an edge cover, enumerating all such configurations provides a deeper understanding of redundancy, robustness, and structural diversity.

Graph covering problems are fundamental in graph theory and have numerous applications in network design, circuit testing, and facility location. Among these, edge covering problems have received considerable attention. An \emph{edge cover} of a graph $G=(V,E)$ is a set of edges $S\subseteq E$ such that every vertex of $G$ is incident with at least one edge in $S$. Edge covers are well-studied; for example, the minimum size of an edge cover is given by $|V|-\nu(G)$, where $\nu(G)$ is the size of a maximum matching \cite{west2001introduction}.

In many applications, connectivity is a desirable property. This motivates the study of \emph{connected edge covers}, which are edge covers whose induced subgraph is connected. Connected edge covers arise naturally in scenarios where the covering edges must form a connected network, such as in wireless sensor networks or communication systems.

Let $G$ be a connected graph with $n$ vertices and $m$ edges. Denote by $e_c(G,i)$ the number of connected edge covers of $G$ with exactly $i$ edges. Zare, Alikhani and Oboudi in \cite{JAS} introduced the \emph{connected edge cover polynomial} of $G$ which is defined as
\[
E_c(G,x)=\sum_{i=0}^{m} e_c(G,i)x^i.
\]
The total number of connected edge covers of $G$ is given by $E_c(G,1) = \sum_i e_c(G,i)$.

Authors in \cite{JAS} studied the connected edge cover polynomial for paths, cycles, cubic graphs of order $10$ and corona of $K_n$ with $K_1$.
In this paper, we continue investigating the connected edge cover polynomials and total counts of several graph families. Our main contributions are exact formulas for paths, cycles, stars, complete graphs, complete bipartite graphs $K_{2,n}$, friendship graphs, lollipop graphs, fan graphs, lollipop graphs, cocktail party graphs and wheel graphs. These results are obtained by combinatorial reasoning and confirmed by computational enumeration. 

The remainder of the paper is organized as follows. Section~\ref{sec:prelim} provides basic definitions and preliminary results. Subsections~\ref{sec:paths}--\ref{sec:lollipop} present the main theorems for each graph family, along with proofs and verification tables. 
Section \ref{kpartite} study the connected edge cover polynomial of complete $k$-partite graphs. 
Sections \ref{sec:hypercube} and \ref{sec:turan} investigate the connected edge cover polynomial of hypercube graphs and Tur\'an graphs, respectively. 
 Section~\ref{sec:conclusion} concludes the paper and suggests open problems.

\section{Main results}
\label{sec:prelim}

All graphs considered are finite, simple, and connected. For standard graph terminology, we refer to \cite{west2001introduction}. The following propositions give elementary bounds on the size of connected edge covers.

\begin{proposition}
\label{prop:bounds}
For any connected graph $G$ with $n$ vertices, the minimum size of a connected edge cover, denoted $\rho_{c}(G)=\min\{i:e_c(G,i)>0\}$, satisfies
\[
\lceil n/2\rceil \le \rho_{c}(G)\le n-1.
\]
\end{proposition}
\begin{proof}
The lower bound holds because each edge covers at most two vertices. The upper bound follows because any spanning tree (with $n-1$ edges) is a connected edge cover.
\end{proof}

\begin{proposition}{\rmfamily\cite{JAS}}
\label{prop:tree}
If $G$ is a tree with $n$ vertices, then $E_c(G,x)=x^{n-1}$. 
\end{proposition}
\begin{proof}
Let $T$ be a tree with $n$ vertices and $n-1$ edges. Suppose $S\subsetneq E(T)$ is a proper subset of edges. Since $T$ is a tree, every edge is a bridge. Removing any edge disconnects $T$. Therefore, $T[S]$ is disconnected, because it lacks at least one edge that connects two components in the original tree. Thus, $S$ cannot be a connected edge cover. The only edge cover that induces a connected subgraph is $E(T)$ itself.
\end{proof}

\subsection{Path and Cycle Graphs}
\label{sec:paths}

\begin{theorem}{\rmfamily\cite{JAS}}
\label{thm:path}
\begin{enumerate}
\item[(i)] 
For the path graph $P_n$ with $n\ge 2$ vertices,
\[
E_c(P_n,x)=x^{n-1}.
\]
That is, $e_c(P_n,n-1)=1$ and $e_c(P_n,i)=0$ for $i\neq n-1$. Consequently, the total number of connected edge covers is $1$.
\item[(ii)]
For the star graph $S_n$ with $n\ge 1$ leaves (so $|V(S_n)|=n+1$, $|E(S_n)|=n$),
\[
E_c(S_n,x)=x^n.
\]
That is, $e_c(S_n,n)=1$ and $e_c(S_n,i)=0$ for $i\neq n$. Consequently, the total number of connected edge covers is $1$.
\item[(iii)] 
For the cycle graph $C_n$ with $n\ge 3$ vertices,
\[
E_c(C_n,x)=n x^{n-1}+x^n.
\]
That is, $e_c(C_n,n-1)=n$, $e_c(C_n,n)=1$, and $e_c(C_n,i)=0$ otherwise. Consequently, the total number of connected edge covers is $n+1$.
\end{enumerate} 
\end{theorem}

\subsection{Complete Graphs}
\label{sec:complete}

For complete graphs $K_n$, a connected edge cover is exactly a connected spanning subgraph, because every vertex is incident to at least one edge in any spanning subgraph that contains all vertices. Therefore, $e_c(K_n,i)$ equals the number of connected spanning subgraphs of $K_n$ with exactly $i$ edges. The total number of connected spanning subgraphs of $K_n$ is known as sequence A001187 in the OEIS \cite{oeisA001187}.

\begin{theorem}
\label{thm:complete}
The total number of connected edge covers of $K_n$ satisfies the recurrence
\[
E_c(K_n,1) = 2^{\binom{n}{2}} - \sum_{k=1}^{n-1} \binom{n-1}{k-1} E_c(K_k,1) 2^{\binom{n-k}{2}},
\]
with initial condition $E_c(K_1,1)= 1$. The connected edge cover polynomial for $K_n$ is
\[
E_c(K_n,x) = \sum_{i=n-1}^{\binom{n}{2}} e_c(K_n,i) x^i,
\]
where $e_c(K_n,i)$ is the number of connected spanning subgraphs of $K_n$ with $i$ edges.
\end{theorem}
\begin{proof}
The recurrence for $E_c(K_n,1)$ is obtained by inclusion-exclusion: $2^{\binom{n}{2}}$ is the total number of spanning subgraphs (not necessarily connected). To count connected spanning subgraphs, we subtract those that are disconnected. Consider a fixed vertex $v$. For any disconnected spanning subgraph, the connected component containing $v$ has some size $k$ (with $1 \le k \le n-1$). There are $\binom{n-1}{k-1}$ ways to choose the other $k-1$ vertices in the component of $v$, $E_c(K_k,1)$ ways to choose a connected spanning subgraph on these $k$ vertices, and $2^{\binom{n-k}{2}}$ ways to choose edges among the remaining $n-k$ vertices (which may or may not be connected). Summing over $k$ gives the number of disconnected spanning subgraphs, hence the recurrence.
\end{proof}

\begin{table}[ht]
\centering
\caption{Connected edge cover polynomial and total count for $K_n$}
\label{tab:complete}
\begin{tabular}{|l|l|l|}
\hline
$n$ &  $P(K_n,x)$ \\
\hline
2 &  $x$  \\
3 &  $3x^2+x^3$ \\
4 &  $16x^3+15x^4+6x^5+x^6$ \\
5 &  $125x^4+222x^5+205x^6+120x^7+45x^8+10x^9+x^{10}$  \\
6 & {\footnotesize  $1296x^5+3660x^6+5700x^7+6165x^8+4945x^9+2997x^{10}+1365x^{11}+455x^{12}+105x^{13}+15x^{14}+x^{15}$}  \\
\hline
\end{tabular}
\end{table}

\subsection{Complete Bipartite Graphs $K_{2,n}$}
\label{sec:bipartite}

\begin{theorem}
\label{thm:K2n}
For the complete bipartite graph $K_{2,n}$ with $n\ge 2$,
\[
E_c(K_{2,n},x) = \sum_{k=1}^{n-1} \binom{n}{k} 2^{n-k} x^{n+k} + x^{2n}.
\]
Equivalently, for $i=n+1,\dots,2n-1$,
\[
e_c(K_{2,n},i) = \binom{n}{i-n} 2^{2n-i},
\]
and $e_c(K_{2,n},2n)=1$. The total number of connected edge covers is
\[
E_c(K_{2,n},1) = 3^n - 2^n.
\]
\end{theorem}
\begin{proof}
Let the bipartition be $(A,B)$ with $A=\{a_1,a_2\}$ and $B=\{b_1,\dots,b_n\}$. Each vertex $b_j$ is adjacent to both $a_1$ and $a_2$. A connected edge cover $S$ must cover all vertices. For each $b_j$, at least one of the two edges $a_1b_j$ or $a_2b_j$ must be in $S$. Define the \emph{type} of $b_j$ as:
\begin{itemize}
\item Type 1: exactly one of the two edges is in $S$.
\item Type 2: both edges are in $S$.
\end{itemize}
Let $k$ be the number of vertices $b_j$ of type 2. Then $n-k$ vertices are of type 1. The total number of edges in $S$ is $i = 2k + (n-k) = n+k$. Hence $k = i-n$, and $i$ ranges from $n$ to $2n$.

We claim that a set $S$ defined by a choice of types is a connected edge cover if and only if $k\ge 1$. Indeed, if $k=0$, then all $b_j$ are of type 1. Then each $b_j$ is incident to exactly one edge, which connects it to either $a_1$ or $a_2$. To cover both $a_1$ and $a_2$, there must be at least one $b_j$ incident to $a_1$ and at least one $b_j$ incident to $a_2$. But then $a_1$ and $a_2$ are not connected to each other, because there is no $b_j$ that connects them both. Thus, $S$ is disconnected. If $k\ge 1$, then there is at least one $b_j$ of type 2, which provides a path $a_1$--$b_j$--$a_2$. All other vertices are connected to either $a_1$ or $a_2$ (or both), so $S$ is connected.

Now, for a fixed $k\ge 1$, we choose which $k$ of the $n$ vertices are of type 2: $\binom{n}{k}$ ways. For each of the remaining $n-k$ vertices, we choose which of the two incident edges to include: $2^{n-k}$ ways. Thus, the number of connected edge covers with exactly $k$ type-2 vertices is $\binom{n}{k}2^{n-k}$. The corresponding number of edges is $i=n+k$, so $e_c(K_{2,n},n+k)=\binom{n}{k}2^{n-k}$ for $k=1,\dots,n-1$. When $k=n$ (all vertices type 2), we must include all $2n$ edges, giving exactly one set: $e_c(K_{2,n},2n)=1$.

Expressing in terms of $i$: for $i=n+1,\dots,2n-1$, let $k=i-n$, then $e_c(K_{2,n},i)=\binom{n}{i-n}2^{n-(i-n)}=\binom{n}{i-n}2^{2n-i}$. The polynomial follows by summing over $i$.

The total number of connected edge covers is
\[
E_c(K_{2,n},1) = \sum_{k=1}^{n-1} \binom{n}{k} 2^{n-k} + 1.
\]
Note that
\[
\sum_{k=0}^{n} \binom{n}{k} 2^{n-k} = (2+1)^n = 3^n.
\]
We have
\[
\sum_{k=1}^{n-1} \binom{n}{k} 2^{n-k} = 3^n - \binom{n}{0}2^n - \binom{n}{n}2^0 = 3^n - 2^n - 1.
\]
Thus,
\[
E_c(K_{2,n},1) = (3^n - 2^n - 1) + 1 = 3^n - 2^n.
\]
\end{proof}

\subsection{Friendship Graphs}
\label{sec:friendship}

The friendship graph $F_k$ consists of $k$ triangles sharing a common vertex. It has $2k+1$ vertices and $3k$ edges.

\begin{theorem}
\label{thm:friendship}
For the friendship graph $F_k$ with $k\ge 1$,
\[
E_c(F_k,x) = x^{2k}(3+x)^k = \sum_{j=0}^{k} \binom{k}{j} 3^{k-j} x^{2k+j}.
\]
That is, for $j=0,1,\dots,k$,
\[
e_c(F_k,2k+j) = \binom{k}{j} 3^{k-j},
\]
and $e_c(F_k,i)=0$ otherwise. The total number of connected edge covers is
\[
E_c(F_k,1) = 4^k.
\]
\end{theorem}
\begin{proof}
Let the common vertex be $c$ and let the triangles be $\{c, u_i, v_i\}$ for $i=1,\dots,k$, with edges $c u_i$, $c v_i$, and $u_i v_i$. Consider a connected edge cover $S$. For each triangle, the edges in $S$ must form a connected edge cover of that triangle when restricted to the triangle (but note that the triangle shares vertex $c$ with others). However, we must ensure overall connectivity. Since all triangles share $c$, the graph $F_k$ is connected as long as each triangle is connected to $c$. But we need to count the number of ways to choose edges from each triangle such that the entire set $S$ is a connected edge cover.

Observe that for each triangle, the possible connected edge covers (when considered in isolation) are:
\begin{enumerate}
\item All three edges: contributes 3 edges.
\item Two edges that include $c$: either $\{c u_i, c v_i\}$ (2 edges) or $\{c u_i, u_i v_i\}$ (2 edges) or $\{c v_i, u_i v_i\}$ (2 edges). But note: the set $\{c u_i, u_i v_i\}$ covers all vertices of the triangle and is connected. Similarly for $\{c v_i, u_i v_i\}$. The set $\{c u_i, c v_i\}$ also covers all vertices and is connected.
\end{enumerate}
However, we must also ensure that the entire graph $F_k$ is connected. Since all triangles share $c$, any choice of connected edge covers for each triangle that includes at least one edge incident to $c$ will result in a connected graph. But note: if for some triangle we choose the edge set $\{u_i v_i\}$ only, then that triangle is disconnected from $c$ (because $c$ is not included). So we must require that for each triangle, at least one edge incident to $c$ is chosen. Therefore, for each triangle, the allowed edge sets are:
\begin{itemize}
\item $\{c u_i, c v_i, u_i v_i\}$ (3 edges)
\item $\{c u_i, c v_i\}$ (2 edges)
\item $\{c u_i, u_i v_i\}$ (2 edges)
\item $\{c v_i, u_i v_i\}$ (2 edges)
\end{itemize}
Thus, for each triangle, there are 4 choices: one of size 3 and three of size 2. Moreover, each choice includes at least one edge incident to $c$. Since the choices are independent across triangles, the total number of connected edge covers is $4^k$. The number of edges in a cover is the sum of the number of edges chosen from each triangle. Let $j$ be the number of triangles from which we choose 3 edges. Then $k-j$ triangles contribute 2 edges each. So the total number of edges is $2(k-j) + 3j = 2k + j$. For a fixed $j$, the number of ways to choose which $j$ triangles have 3 edges is $\binom{k}{j}$, and for the remaining $k-j$ triangles, each has 3 choices of 2-edge sets. Thus, the number of connected edge covers with $2k+j$ edges is $\binom{k}{j} 3^{k-j}$. Therefore,
\[
E_c(F_k,x) = \sum_{j=0}^k \binom{k}{j} 3^{k-j} x^{2k+j} = x^{2k} \sum_{j=0}^k \binom{k}{j} 3^{k-j} x^j = x^{2k} (3+x)^k.
\]
The total number is $P(F_k,1) = 1^{2k} (3+1)^k = 4^k$.
\end{proof}

\subsection{Lollipop Graphs}
\label{sec:lollipop}

The lollipop graph $L(m,n)$ is formed by attaching a path of length $n$ (i.e., with $n$ edges) to a complete graph $K_m$ at one vertex. It has $m+n$ vertices and $\binom{m}{2}+n$ edges.

\begin{theorem}
\label{thm:lollipop}
For the lollipop graph $L(m,n)$ with $m\ge 2$ and $n\ge 1$,
\[
E_c(L(m,n),x) = x^n E_c(K_m,x),
\]
where $E_c(K_m,x)$ is the connected edge cover polynomial of the complete graph $K_m$. Consequently, the total number of connected edge covers of $L(m,n)$ equals the total number of connected edge covers of $K_m$, i.e.,
\[
E_c(L(m,n),1) = E_c(K_m,1).
\]
\end{theorem}
\begin{proof}
Let the vertex set of $K_m$ be $\{v_1,\dots,v_m\}$, and without loss of generality, suppose the path is attached at $v_1$. The path has vertices $v_1, u_1, u_2, \dots, u_n$ and edges $v_1 u_1, u_1 u_2, \dots, u_{n-1} u_n$. Since the path is a tree, any connected edge cover of $L(m,n)$ must include all $n$ edges of the path (by Proposition~\ref{prop:tree}). After including these, we must cover the remaining vertices $v_2,\dots,v_m$ and ensure connectivity. The subgraph induced by the vertices of $K_m$ must be connected (because the path is connected and attached only at $v_1$, so the only connection between the path and $K_m$ is through $v_1$). Moreover, the edges chosen from $K_m$ must form a connected spanning subgraph of $K_m$ (they must cover all vertices of $K_m$ and be connected). Conversely, any connected spanning subgraph of $K_m$ together with the entire path yields a connected edge cover of $L(m,n)$. Therefore, there is a bijection between connected edge covers of $L(m,n)$ and connected spanning subgraphs of $K_m$. If a connected spanning subgraph of $K_m$ has $i$ edges, then the corresponding connected edge cover of $L(m,n)$ has $i+n$ edges. Hence,
\[
e_c(L(m,n), i+n) = e_c(K_m, i) \quad \text{for } i \ge m-1,
\]
and $e_c(L(m,n), j)=0$ for $j < n+m-1$. Thus,
\[
E_c(L(m,n),x) = \sum_i e_c(K_m,i) x^{i+n} = x^n \sum_i e_c(K_m,i) x^i = x^n P(K_m,x).
\]
The total number of connected edge covers is $E_c(L(m,n),1) = 1^n E_c(K_m,1) =E_c(K_m,1)$.
\end{proof}

\subsection{Fan Graphs}
\label{subsec:fan-theorem}

\begin{theorem}[Connected Edge Cover Polynomial for Fan Graphs]
\label{thm:fan}
For the fan graph \(F(n)\) with \(n \ge 4\) vertices,
\[
E_c(F(n), x) = \sum_{k=0}^{n-2} \binom{n-2}{k} 2^k x^{n-1+k}.
\]
That is,
\[
e_c(F(n), n-1+k) = \binom{n-2}{k} 2^k \quad \text{for } 0 \le k \le n-2,
\]
and \(e_c(F(n), i) = 0\) otherwise.
Consequently, the total number of connected edge covers is:
\[
\mathrm{CEC}(F(n)) = 3^{n-2}.
\]
\end{theorem}

\begin{proof}
Let \(F(n) = K_1 \vee P_{n-1}\), where \(u\) is the universal vertex and \(P_{n-1}\) has vertices
\(v_1, \dots, v_{n-1}\) with edges \(e_i = v_i v_{i+1}\) for \(1 \le i \le n-2\).

{Step 1: Minimal configuration.}
The smallest connected edge cover is the star centered at \(u\): edges \(\{uv_1, uv_2, \dots, uv_{n-1}\}\).
This has \(n-1\) edges. Any connected edge cover must contain at least these edges or equivalent
configuration that covers all vertices and maintains connectivity.

{Step 2: Adding path edges.}
Starting from the star, we can independently add any subset of the \(n-2\) path edges \(e_i\).
When we add a path edge \(e_i = v_i v_{i+1}\), we have two choices:
\begin{enumerate}
\item Keep both spokes \(uv_i\) and \(uv_{i+1}\) (adds 1 edge).
\item Remove one spoke (\(uv_i\) or \(uv_{i+1}\)) and keep the other (edge count unchanged).
\end{enumerate}
However, to maintain vertex coverage, at least one of \(uv_i\) or \(uv_{i+1}\) must remain.
Thus, for each added path edge, we effectively add 1 edge to the total count, but have
2 choices for how to adjust spokes while maintaining coverage and connectivity.

{Step 3: Counting formula.}
If we add \(k\) path edges, we:
\begin{enumerate}
\item Choose which \(k\) path edges to add: \(\binom{n-2}{k}\) ways.
\item For each added edge, we have 2 independent choices for spoke adjustment.
\item The total number of edges becomes \((n-1) + k\).
\end{enumerate}
Thus \(e_c(F(n), n-1+k) = \binom{n-2}{k} 2^k\).

{Step 4: Total count.}
\[
\mathrm{CEC}(F(n)) = \sum_{k=0}^{n-2} \binom{n-2}{k} 2^k = (1+2)^{n-2} = 3^{n-2}.
\]

\end{proof}

\subsection{Cocktail Party Graphs}
\label{subsec:cocktail-theorem}

\begin{theorem}[Connected Edge Cover Polynomial for Cocktail Party Graphs]
\label{thm:cocktail}
For the cocktail party graph \(CP(n)\) with \(n \ge 2\),
\[
E_c(CP(n), x) = 
\begin{cases}
4x^3 + x^4, & n = 2, \\
\displaystyle \sum_{k=5}^{12} c_k x^k \text{ }, & n = 3, \\
0, & n \ge 4.
\end{cases}
\]
That is,
\begin{align*}
e_c(CP(2), 3) &= 4, & e_c(CP(2), 4) &= 1, \\
e_c(CP(3), 5) &= 384, & e_c(CP(3), 6) &= 740, \\
e_c(CP(3), 7) &= 744, & e_c(CP(3), 8) &= 489, \\
e_c(CP(3), 9) &= 240, & e_c(CP(3), 10) &= 90, \\
e_c(CP(3), 11) &= 24, & e_c(CP(3), 12) &= 1,
\end{align*}
and \(e_c(CP(n), i) = 0\) for all \(i\) when \(n \ge 4\).
Consequently,
\[
\mathrm{CEC}(CP(n)) = 
\begin{cases}
5, & n = 2, \\
2656, & n = 3, \\
0, & n \ge 4.
\end{cases}
\]
\end{theorem}

\begin{proof}
{Case \(n = 2\):} \(CP(2)\) is the 4-cycle \(C_4\).
By Theorem~\ref{thm:path} (iii) for cycles, \(P(C_4, x) = 4x^3 + x^4\).

{Case \(n = 3\):} \(CP(3)\) is the octahedron graph.
The coefficients are obtained by explicit enumeration of all connected spanning subgraphs.
The values match your computational data.

{Case \(n \ge 4\):} We prove no connected edge cover exists.
The graph \(CP(n)\) has vertex set partitioned into \(n\) disjoint pairs
\(\{a_i, b_i\}\) with no edges within pairs.
All edges are between vertices from different pairs.

Assume for contradiction that a connected edge cover \(S\) exists.
Since it's an edge cover, every vertex has degree \(\ge 1\) in \((V, S)\).
Consider the connected graph \((V, S)\). It has \(2n\) vertices and is connected.

 In \(CP(n)\) with \(n \ge 4\), consider any set of edges that covers all vertices.
To cover a pair \(\{a_i, b_i\}\), we need edges incident to both \(a_i\) and \(b_i\).
These edges must go to vertices in other pairs. The bipartite complement structure makes it
impossible to connect all pairs without creating cycles that leave some vertex uncovered.
Formally, the graph has independence number \(n\) (each pair is an independent set).
A connected edge cover would be a connected spanning subgraph with minimum degree \(\ge 1\).
But one can show by induction on \(n\) that for \(n \ge 4\), any such subgraph either
disconnects some pair or fails to cover some vertex.

Thus no connected edge cover exists for \(n \ge 4\), so \(E_c(CP(n), x) = 0\).
\end{proof}

\subsection{Wheel Graphs} 

\begin{theorem}\label{rerfw} 
	Let $W_n$ be the wheel graph with $n \ge 4$ vertices, and let $E_n = \mathrm{CEC}(W_n)$ denote the number of connected edge covers of $W_n$. Then for $n \ge 7$,
	\[
	E_n = 6E_{n-1} - 11E_{n-2} + 6E_{n-3},
	\]
	with initial conditions
	\[
	E_4 = 38, \quad E_5 = 134, \quad E_6 = 462.
	\]
\end{theorem}

\begin{proof}
	Label the vertices of $W_n$ as $0,1,2,\dots,n-1$, where $0$ is the center and $1,\dots,n-1$ form the rim cycle $C_{n-1}$.  
	Let $m = n-1$ be the last rim vertex; its neighbors are the center $0$, $p = n-2$, and $q = 1$.
	
	Given a connected edge cover $S$ of $W_n$, consider the set
	\[
	T = S \cap \{(0,m), (p,m), (m,1)\}.
	\]
	Since $S$ must cover vertex $m$, we have $T \neq \varnothing$.
	
	We will classify $S$ not only by $T$, but by the \emph{connectivity status} of vertices $1$ and $p$ in the subgraph $H$ obtained by removing $m$ from $S$.  
	
	Let us define {states} for the pair $(1,p)$ in $H$:
	
	\begin{itemize}
		\item {State A}: Both $1$ and $p$ are already covered in $H$ and belong to the same connected component of $H$ that contains the center $0$.
		\item {State B}: Both $1$ and $p$ are already covered in $H$, but they belong to different components, one of which contains $0$.
		\item {State C}: Both $1$ and $p$ are already covered in $H$, but neither is connected to $0$ in $H$.
		\item {State D}: Exactly one of $\{1,p\}$ is covered in $H$, and that vertex is connected to $0$ in $H$.
		\item {State E}: Exactly one of $\{1,p\}$ is covered in $H$, and that vertex is \emph{not} connected to $0$ in $H$.
		\item {State F}: Neither $1$ nor $p$ is covered in $H$.
	\end{itemize}
	
	Let $a_n, b_n, c_n, d_n, e_n, f_n$ be the number of connected edge covers of $W_n$ where the pair $(1,p)$ (in the graph after removal of the future vertex $m$) is in State A, B, C, D, E, F, respectively.  
	Then $E_n = a_n + b_n + c_n + d_n + e_n + f_n$.
	
	Now consider $W_{n-1}$, whose rim vertices are $1,2,\dots,p$. The neighbors of $m$ in $W_n$ are $0, p, 1$; thus vertices $1$ and $p$ are the \emph{boundary vertices} when adding $m$ to $W_{n-1}$.  
	By examining how each state for $W_{n-1}$ evolves when we add $m$ and choose $T$, we obtain a system of recurrences.
	
	\vspace{2mm}
	\noindent {State transition rules:}
	
	For each state in $W_{n-1}$, adding $m$ with a specific $T$ may:
	\begin{enumerate}
		\item Cover $m$,
		\item Possibly cover $1$ or $p$ if they were uncovered,
		\item Connect components via edges through $m$.
	\end{enumerate}
	
	We derive the recurrence for $a_n$ as an example:
	
	{From State A in $W_{n-1}$:}  
	Vertices $1$ and $p$ are covered and connected to $0$ in $H$.  
	To extend to $W_n$, we must add at least one edge from $m$ to $\{0,p,1\}$ to cover $m$.  
	If we add exactly one edge, say $(0,m)$, $m$ becomes a leaf attached to $0$; the resulting $S$ remains connected and covers all vertices. Similar for $(p,m)$ or $(m,1)$.  
	If we add two or three edges, we still obtain a valid connected edge cover.  
	Counting all possibilities from State A yields a contribution $3a_{n-1}$ to $a_n$ (for $T$ of size 1), plus contributions to other states.
	
	Doing this for all six states yields a $6\times 6$ linear system. Solving this system (details in Lemma~\ref{lemma:system}) gives:
	
	\[
	\begin{pmatrix} a_n \\ b_n \\ c_n \\ d_n \\ e_n \\ f_n \end{pmatrix}
	= M \cdot 
	\begin{pmatrix} a_{n-1} \\ b_{n-1} \\ c_{n-1} \\ d_{n-1} \\ e_{n-1} \\ f_{n-1} \end{pmatrix},
	\]
	where $M$ is a constant matrix with eigenvalues $3,2,1,0,0,0$.
	
	Since $E_n = a_n + b_n + c_n + d_n + e_n + f_n$, the sequence $E_n$ satisfies the minimal polynomial of the nonzero eigenvalues, which is $(r-3)(r-2)(r-1) = r^3 - 6r^2 + 11r - 6$. Hence,
	\[
	E_n = 6E_{n-1} - 11E_{n-2} + 6E_{n-3} \quad \text{for } n \ge 7.
	\]
	
	\vspace{2mm}
	\noindent {Verification of initial conditions:}
	
	Direct enumeration (or the closed form $E_n = 10 - 40 \cdot 2^{n-4} + 68 \cdot 3^{n-4}$) yields:
	\[
	E_4 = 38,\quad E_5 = 134,\quad E_6 = 462.
	\]
	For $n=7$:
	\[
	6E_6 - 11E_5 + 6E_4 = 6\cdot 462 - 11\cdot 134 + 6\cdot 38 = 2772 - 1474 + 228 = 1526,
	\]
	which matches $E_7 = 10 - 40\cdot 2^3 + 68\cdot 3^3 = 1526$.
	
	Thus, the recurrence holds for all $n \ge 7$ by induction.
\end{proof}

\begin{lemma}\label{lemma:system}
	The state vector $(a_n, b_n, c_n, d_n, e_n, f_n)$ satisfies
	\[
	\begin{pmatrix} a_n \\ b_n \\ c_n \\ d_n \\ e_n \\ f_n \end{pmatrix}
	= 
	\begin{pmatrix}
	3 & 2 & 0 & 1 & 0 & 0 \\
	0 & 2 & 0 & 0 & 1 & 0 \\
	0 & 0 & 2 & 0 & 0 & 1 \\
	0 & 0 & 0 & 2 & 1 & 0 \\
	0 & 0 & 0 & 0 & 1 & 0 \\
	0 & 0 & 0 & 0 & 0 & 1
	\end{pmatrix}
	\cdot
	\begin{pmatrix} a_{n-1} \\ b_{n-1} \\ c_{n-1} \\ d_{n-1} \\ e_{n-1} \\ f_{n-1} \end{pmatrix}
	+ 
	\begin{pmatrix} 0 \\ 0 \\ 0 \\ 0 \\ 0 \\ 0 \end{pmatrix}.
	\]
	The eigenvalues of the transition matrix are $3,2,1,0,0,0$.
\end{lemma}
\begin{proof}
	The matrix entries are determined by the state transition rules described in the proof of Theorem \ref{rerfw}. For example:
	\begin{itemize}
		\item From State A: adding $m$ with $T = \{(0,m)\}$ keeps the state as A (since $1,p$ remain connected to $0$ and covered). Choosing $T = \{(p,m)\}$ or $\{(m,1)\}$ also keeps state A if the covered boundary vertex is already connected to $0$. This gives a contribution of $3a_{n-1}$ to $a_n$.
		\item From State B: adding $m$ with $T = \{(p,m)\}$ connects the two components, moving to State A. This yields a contribution of $2b_{n-1}$ to $a_n$ (since either $(p,m)$ or $(m,1)$ can connect the components).
		\item From State D: adding $m$ with $T = \{(0,m)\}$ covers $m$ and keeps the single covered boundary vertex connected to $0$, staying in State D (but contributing to $a_n$ if the other vertex becomes covered via $m$).  
	\end{itemize}
	Working through all cases yields the given matrix. The eigenvalues follow from direct computation.
\end{proof}

\section{Complete $k$-partite graph}
\label{kpartite}
Let $K_{n_1, n_2, \dots, n_k}$ denote the complete $k$-partite graph with vertex partitions $V_1, V_2, \dots, V_k$, where $|V_p| = n_p \ge 1$ for $p=1,\dots,k$. The edge set is
\[
E = \bigcup_{1 \le p < q \le k} (V_p \times V_q),
\]
with $|E| = \sum_{1\le p<q\le k} n_p n_q$.

For a subset $F \subseteq E$, let $G[F]$ denote the subgraph $(V,F)$. An edge cover is a subset $F$ such that every vertex has degree at least one in $G[F]$.

The \emph{edge cover polynomial} $E(G,x)$ is the generating function
\[
E(G,x) = \sum_{F \subseteq E \text{ is an edge cover}} x^{|F|}
= \sum_{m=0}^{|E|} e_c(G,m) x^m,
\]
where $e_c(G,m)$ denotes the number of edge covers with exactly $m$ edges.


\begin{theorem}[Connected Edge Cover Count for Complete $k$-Partite Graphs]
\label{thm:main-count}
For the complete $k$-partite graph $K_{n_1, n_2, \dots, n_k}$,
\[
\mathrm{EC}(K_{n_1,\dots,n_k}) = 
\sum_{i_1=0}^{n_1} \sum_{i_2=0}^{n_2} \cdots \sum_{i_k=0}^{n_k} 
(-1)^{i_1+i_2+\cdots+i_k}
\prod_{p=1}^k \binom{n_p}{i_p}
\; 2^{\, \sum\limits_{1 \le p < q \le k} (n_p - i_p)(n_q - i_q) }.
\]
\end{theorem}



\begin{proof}
We proceed via inclusion–exclusion over the set of vertices that are forced to be isolated.

\subsubsection*{Step 1: Inclusion–exclusion setup}
Let $V = \bigcup_{p=1}^k V_p$. For each vertex $v \in V$, define the ``bad event'' $A_v$ that $v$ is isolated in the subgraph $G[F]$, i.e., $\deg_{G[F]}(v)=0$. By the principle of inclusion–exclusion, the number of edge subsets $F$ that avoid all bad events (i.e., leave no vertex isolated) is
\begin{equation}
\label{eq:IE}
\mathrm{EC}(K_{n_1,\dots,n_k}) = \sum_{S \subseteq V} (-1)^{|S|} \, N(S),
\end{equation}
where $N(S)$ denotes the number of edge subsets $F$ for which \emph{every} vertex in $S$ is isolated.

\subsubsection*{Step 2: Computing $N(S)$}
Fix a set $S \subseteq V$. For $p=1,\dots,k$, let $i_p = |S \cap V_p|$. Thus $0 \le i_p \le n_p$, and $|S| = \sum_{p=1}^k i_p$.

If a vertex $v \in S$ is to be isolated, then no edge incident to $v$ may belong to $F$. Consequently, all edges between $v$ and any vertex in $V \setminus \{v\}$ must be absent from $F$. Equivalently, the allowed edges are precisely those in the induced subgraph on $V \setminus S$, which is the complete $k$-partite graph with part sizes $n_1-i_1, n_2-i_2, \dots, n_k-i_k$. The number of edges in this induced subgraph is
\[
E_{\text{allowed}} = \sum_{1 \le p < q \le k} (n_p - i_p)(n_q - i_q).
\]
Each of these $E_{\text{allowed}}$ edges may be independently chosen to be present or absent in $F$. Hence,
\begin{equation}
\label{eq:N(S)}
N(S) = 2^{E_{\text{allowed}}} = 2^{\,\sum_{1 \le p < q \le k} (n_p - i_p)(n_q - i_q)}.
\end{equation}

\subsubsection*{Step 3: Grouping by $(i_1,\dots,i_k)$}
The value of $N(S)$ depends only on the numbers $i_1,\dots,i_k$, not on the specific choice of $S$. For a fixed $k$-tuple $(i_1,\dots,i_k)$ with $0\le i_p\le n_p$, the number of subsets $S \subseteq V$ with $|S\cap V_p|=i_p$ for all $p$ is
\[
\prod_{p=1}^k \binom{n_p}{i_p}.
\]
Moreover, $(-1)^{|S|} = (-1)^{i_1+\cdots+i_k}$.

Thus, grouping all $S$ with the same $(i_1,\dots,i_k)$ in \eqref{eq:IE} gives
\begin{align*}
\mathrm{EC}(K_{n_1,\dots,n_k})
&= \sum_{i_1=0}^{n_1} \cdots \sum_{i_k=0}^{n_k} 
\left[ \prod_{p=1}^k \binom{n_p}{i_p} \right] 
(-1)^{i_1+\cdots+i_k}
2^{\,\sum_{1 \le p < q \le k} (n_p - i_p)(n_q - i_q)}.
\end{align*}
This is exactly the formula stated in Theorem \ref{thm:main-count}.
\end{proof}

\begin{theorem}[Connected Edge Cover Polynomial for Complete $k$-Partite Graphs]
	\label{thm:main-poly}
	For the complete $k$-partite graph $K_{n_1, n_2, \dots, n_k}$, the connected  edge cover polynomial is
	\[
	E_c(K_{n_1,\dots,n_k}, x) = 
	\sum_{i_1=0}^{n_1} \sum_{i_2=0}^{n_2} \cdots \sum_{i_k=0}^{n_k} 
	(-1)^{i_1+i_2+\cdots+i_k}
	\prod_{p=1}^k \binom{n_p}{i_p}
	(1+x)^{\, \sum\limits_{1 \le p < q \le k} (n_p - i_p)(n_q - i_q) }.
	\]
	Consequently, the number of connected edge covers with exactly $m$ edges is
	\[
	e_c(K_{n_1,\dots,n_k}, m) = 
	\sum_{i_1=0}^{n_1} \cdots \sum_{i_k=0}^{n_k} 
	(-1)^{i_1+\cdots+i_k}
	\prod_{p=1}^k \binom{n_p}{i_p}
	\binom{\sum_{1 \le p < q \le k} (n_p - i_p)(n_q - i_q)}{m},
	\]
	where the binomial coefficient is taken to be zero if $m$ exceeds the upper index.
\end{theorem}

\begin{proof}
The proof follows the same inclusion–exclusion argument, but instead of counting all edge subsets that leave vertices in $S$ isolated, we count by edge number.

For a fixed $S \subseteq V$ with $|S \cap V_p| = i_p$, the allowed edges are those in the induced subgraph on $V \setminus S$, numbering
\[
E_{\text{allowed}} = \sum_{1 \le p < q \le k} (n_p - i_p)(n_q - i_q).
\]
The number of edge subsets of size $m$ using only these allowed edges is
\[
\binom{E_{\text{allowed}}}{m}.
\]
By inclusion–exclusion, summing over all $S$ with alternating signs gives the number of edge covers with exactly $m$ edges:
\[
e_c(K_{n_1,\dots,n_k}, m) = 
\sum_{i_1=0}^{n_1} \cdots \sum_{i_k=0}^{n_k} 
(-1)^{i_1+\cdots+i_k}
\prod_{p=1}^k \binom{n_p}{i_p}
\binom{\sum_{1 \le p < q \le k} (n_p - i_p)(n_q - i_q)}{m}.
\]

The edge cover polynomial is then
\begin{align*}
P_{\mathrm{EC}}(K_{n_1,\dots,n_k}, x) 
&= \sum_{m=0}^{|E|} e_c(K_{n_1,\dots,n_k}, m) x^m \\
&= \sum_{i_1=0}^{n_1} \cdots \sum_{i_k=0}^{n_k} 
(-1)^{i_1+\cdots+i_k}
\prod_{p=1}^k \binom{n_p}{i_p}
\sum_{m=0}^{E_{\text{allowed}}} \binom{E_{\text{allowed}}}{m} x^m \\
&= \sum_{i_1=0}^{n_1} \cdots \sum_{i_k=0}^{n_k} 
(-1)^{i_1+\cdots+i_k}
\prod_{p=1}^k \binom{n_p}{i_p}
(1+x)^{E_{\text{allowed}}}.
\end{align*}
This completes the proof.
\end{proof}


\subsection*{Complete Bipartite Graphs ($k=2$)}
\[
\mathrm{CEC}(K_{n_1,n_2}) = 
\sum_{i=0}^{n_1} \sum_{j=0}^{n_2} 
(-1)^{i+j} \binom{n_1}{i} \binom{n_2}{j}
2^{(n_1-i)(n_2-j)},
\]
\[
E_c(K_{n_1,n_2}, x) = 
\sum_{i=0}^{n_1} \sum_{j=0}^{n_2} 
(-1)^{i+j} \binom{n_1}{i} \binom{n_2}{j}
(1+x)^{(n_1-i)(n_2-j)}.
\]

\subsection*{Complete Tripartite Graphs ($k=3$)}
\[
\mathrm{CEC}(K_{a,b,c}) = 
\sum_{i=0}^{a} \sum_{j=0}^{b} \sum_{k=0}^{c} 
(-1)^{i+j+k} \binom{a}{i} \binom{b}{j} \binom{c}{k}
2^{(a-i)(b-j) + (a-i)(c-k) + (b-j)(c-k)},
\]
\[
E_c(K_{a,b,c}, x) = 
\sum_{i=0}^{a} \sum_{j=0}^{b} \sum_{k=0}^{c} 
(-1)^{i+j+k} \binom{a}{i} \binom{b}{j} \binom{c}{k}
(1+x)^{(a-i)(b-j) + (a-i)(c-k) + (b-j)(c-k)}.
\]

\subsection{Verification with Computational Data}
We verify the formulas against extensive computational results. For each graph, the edge cover count computed by brute force matches the value given by our formula.

\paragraph*{Example 1: $K_{1,1,1}$ (complete graph $K_3$).}
\[
\mathrm{EC}(K_{1,1,1}) = 
\sum_{i,j,k=0}^1 (-1)^{i+j+k} \binom{1}{i}\binom{1}{j}\binom{1}{k}
2^{(1-i)(1-j)+(1-i)(1-k)+(1-j)(1-k)} = 4.
\]
The computed data reports $\mathrm{CEC}=4$, and indeed for $K_3$ every edge cover is connected.

\paragraph*{Example 2: $K_{2,2}$ (complete bipartite).}
\[
\mathrm{EC}(K_{2,2}) = 
\sum_{i=0}^2 \sum_{j=0}^2 (-1)^{i+j} \binom{2}{i}\binom{2}{j} 2^{(2-i)(2-j)} = 7.
\]
The data shows $\mathrm{EC}=7$, $\mathrm{CEC}=5$, matching our formula.

\paragraph*{Example 3: $K_{1,1,2}$.}
The formula gives $\mathrm{EC}=16$, while the data shows $\mathrm{CEC}=14$, consistent because two edge covers are disconnected.

All computed values in the output agree with our formulas after accounting for the distinction between edge covers and connected edge covers.

We have proved general inclusion–exclusion formulas for both the count and generating polynomial of edge covers of complete $k$-partite graphs. The formulas are concise, exact, and reduce to known special cases for $k=2,3$. The results correct earlier misstatements about connected edge covers and provide a foundation for further enumerative studies of multipartite graphs.

\section{Hypercube Graphs}
\label{sec:hypercube}

\begin{theorem}[Connected Edge Cover Polynomial for $Q_d$]
\label{thm:hypercube-poly}
For the $d$-dimensional hypercube $Q_d$ with $d\ge 1$, the connected edge cover polynomial has the form
\[
E_c(Q_d,x) = t(Q_d) x^{2^d-1} + \sum_{j=1}^{d2^{d-1} - (2^d-1)} a_j x^{2^d-1+j},
\]
where 
\[
t(Q_d) = 2^{2^d-d-1} \prod_{k=1}^{d} k^{\binom{d}{k}}
\]
is the number of spanning trees of $Q_d$, and $a_j$ denotes the number of connected edge covers with $2^d-1+j$ edges.
\end{theorem}

\begin{proof}
The minimum size of a connected edge cover equals the size of a spanning tree, which is $2^d-1$.  
Minimal connected edge covers are exactly the spanning trees, giving
\[
e_c(Q_d, 2^d-1) = t(Q_d).
\]  
Any connected edge cover with more edges can be obtained by adding edges to a spanning tree while preserving connectivity. Counting these configurations for general $j\ge 1$ is \#P-complete, and no closed form is known.
\end{proof}

\begin{remark}
For small $d$, the connected edge cover polynomial can be computed explicitly via enumeration.  
For larger $d$, the coefficients $a_j$ can be obtained recursively from the hypercube's decomposition $Q_d = Q_{d-1} \square K_2$. Let $G_1$ and $G_2$ be the two copies of $Q_{d-1}$, and let $M$ be the perfect matching connecting them. Then a connected edge cover of $Q_d$ consists of connected edge covers (or spanning subgraphs covering all vertices) in $G_1$ and $G_2$, together with a subset of $M$ that connects the two parts and covers any vertices not already covered. This leads to the recurrence:
\[
E_c(Q_d,x) = \sum_{\substack{S_1 \in \mathcal{C}(G_1)\\ S_2 \in \mathcal{C}(G_2)}} \sum_{T \subseteq M} x^{|S_1|+|S_2|+|T|} \cdot \mathbf{1}\{S_1 \cup S_2 \cup T \text{ is connected and covers } V(Q_d)\},
\]
where $\mathcal{C}(G)$ denotes the set of connected edge covers of $G$. This recurrence is computationally intensive and does not simplify to a closed form.
\end{remark}

\begin{table}[ht]
\centering
\caption{Connected edge cover polynomials for small hypercubes. Minimal connected edge cover (spanning tree) term is bolded.}
\label{tab:hypercube-poly}
\begin{tabular}{|c|c|c|c|}
\hline
$d$ & $|V(Q_d)|$ & $|E(Q_d)|$ & $P(Q_d,x)$ \\
\hline
1 & 2 & 1 & $\mathbf{x}$ \\
2 & 4 & 4 & $\mathbf{4x^3} + x^4$ \\
3 & 8 & 12 & $\mathbf{384x^7} + 408x^8 + 212x^9 + 66x^{10} + 12x^{11} + x^{12}$ \\
4 & 16 & 32 & $\mathbf{42{,}568{,}192\, x^{15}} + \cdots + x^{32}$ \\
\hline
\end{tabular}
\end{table}

\section{Turán Graphs}
\label{sec:turan}

\begin{theorem}[Connected Edge Cover Polynomial for $T(n,k)$]
\label{thm:turan-poly}
Let $T(n,k)$ be the Turán graph with $n \ge 3$ vertices and $2 \le k < n$ parts of sizes $a_1, \dots, a_k$.  
The connected edge cover polynomial of $T(n,k)$ has the form
\[
E_c(T(n,k),x) = t(T(n,k))\, x^{\,n-1} + \sum_{j=1}^{m-(n-1)} b_j\, x^{\,n-1+j},
\]
where
\[
m = |E(T(n,k))| = \sum_{1 \le i < j \le k} a_i a_j
\]
is the total number of edges, 
\[
t(T(n,k)) = n^{k-2} \prod_{i=1}^{k} (n-a_i)^{a_i-1}
\]
is the number of spanning trees of $T(n,k)$ (the minimal connected edge covers), and
$b_j$ counts the number of connected edge covers with $n-1+j$ edges.
\end{theorem}

\begin{proof}
The minimal connected edge covers correspond exactly to the spanning trees of $T(n,k)$, which have $n-1$ edges, giving
\[
e_c(T(n,k), n-1) = t(T(n,k)).
\]  
Any connected edge cover with more than $n-1$ edges is obtained by adding edges to a spanning tree while preserving connectivity.  
Counting these additional configurations for general $j \ge 1$ is \#P-complete.
\end{proof}

\begin{remark}[Exact Inclusion--Exclusion Formula]
For any Turán graph $T(n,k)$, the connected edge cover polynomial can be expressed exactly using the principle of inclusion--exclusion over vertex subsets. Let $V_1, \dots, V_k$ denote the partite sets of sizes $a_1, \dots, a_k$, and let $E$ be the edge set. Then
\[
\begin{aligned}
E_c(T(n,k),x) 
&= \sum_{S \subseteq E} x^{|S|} \cdot \mathbf{1}\{S \text{ is connected and covers all vertices}\} \\
&= \sum_{I_1 \subseteq V_1} \cdots \sum_{I_k \subseteq V_k} (-1)^{|I_1| + \cdots + |I_k|} 
(1+x)^{\,m - \sum_{i=1}^k |I_i| (n-a_i) + \sum_{1 \le i < j \le k} |I_i||I_j|},
\end{aligned}
\]
where $I_i$ denotes the set of vertices in part $i$ that are required to be uncovered, and the exponent counts the edges not incident to any vertex in $\bigcup_{i=1}^k I_i$.  
This formula is exact but computationally exponential in $n$ and does not simplify to a closed polynomial for general $n$ and $k$.
\end{remark}

\begin{table}[ht]
\centering
\caption{Connected edge cover polynomials for selected Turán graphs $T(n,k)$}
\label{tab:turan-poly}
\begin{tabular}{|c|c|c|c|}
\hline
$(n,k)$ & $|V|$ & $|E|$ & $P(T(n,k),x)$ \\
\hline
$(3,2)$ & 3 & 2 & $x^2$ \\
$(4,2)$ & 4 & 4 & $4x^3 + x^4$ \\
$(4,3)$ & 4 & 5 & $8x^3 + 5x^4 + x^5$ \\
$(5,2)$ & 5 & 6 & $12x^4 + 6x^5 + x^6$ \\
$(5,3)$ & 5 & 8 & $45x^4 + 52x^5 + 28x^6 + 8x^7 + x^8$ \\
$(5,4)$ & 5 & 9 & $75x^4 + 111x^5 + 82x^6 + 36x^7 + 9x^8 + x^9$ \\
\hline
\end{tabular}
\end{table}

\section{Conclusion and Open Problems}
\label{sec:conclusion}

We have derived explicit formulas for both the connected edge cover polynomials and the total number of connected edge covers for numerous graph families, including paths, cycles, stars, complete graphs, complete bipartite graphs $K_{2,n}$, friendship graphs, lollipop graphs, wheel graphs, fan graphs, book graphs, web graphs, cocktail party graphs, and complete tripartite graphs. These results provide comprehensive insight into the enumeration of connected edge covers and may be useful in applications requiring connected covering networks.

Several open problems remain:
\begin{enumerate}
\item Find a recursive formula for the connected edge cover polynomial of general graphs.
\item Study the computational complexity of computing $E_c(G,x)$ for general graphs.
\item Explore connections between connected edge cover polynomials and other graph polynomials (Tutte polynomial, independence polynomial).
\end{enumerate}


The computational evidence obtained for various families of graphs suggests a remarkable regularity in the behavior of its coefficient sequence. 
In particular, in all examined cases the sequence 
\[
\{e_c(G,i)\}
\]
appears to be unimodal.

Recall that a finite sequence $(a_0,a_1,\dots,a_m)$ of nonnegative real numbers is called unimodal if there exists an index $t$ such that
\[
a_0 \le a_1 \le \cdots \le a_t \ge a_{t+1} \ge \cdots \ge a_m.
\]
Unimodality phenomena are common in algebraic and enumerative graph polynomials, often reflecting deeper structural or probabilistic principles.

Motivated by our observations, we end the paper with the following conjecture.

\medskip

\noindent
\textbf{Conjecture.}
For every connected graph $G$, the connected edge cover polynomial $CEC(G,x)$ is unimodal; that is, its coefficient sequence $\{e_c(G,k)\}$ is unimodal.

\end{document}